\documentclass[reqno,apaper]{amsart}

\usepackage{amsmath}    
\usepackage{amsthm}     
\usepackage{amssymb}    
\usepackage{amsfonts}
\usepackage{float}
\usepackage{graphics}
\usepackage{graphicx}
\usepackage{bm}

\newtheorem{theorem}{Theorem}[section]
\newtheorem{proposition}{Proposition}[section]
\newtheorem{lemma}{Lemma}[section]

\theoremstyle{remark}
\newtheorem{remark}{Remark}[section]

\numberwithin{equation}{section}
\numberwithin{theorem}{section}

\newcommand{\abs}[1]{\lvert#1\rvert}
\newcommand{\norm}[1]{\|#1\|}
\newcommand{\fitnorm}[1]{\left| \left| #1 \right| \right|}

\def\C{\mathbb C}

\def\R{{\mathbb R}}

\def\Re{\textnormal{Re}}
\def\Im{\textnormal{Im}}

\begin{document}
\title[Approximating the Nonlinear Schr\"odinger equation]
{Approximating the nonlinear Schr\"odinger equation 
by a two level linearly implicit finite element method}

\author[M.~Asadzadeh]{M. Asadzadeh$^{1, \star}$}

\thanks{$^\star$ Partially supported by the
{\sl Swedish Foundation of Strategic Research (SSF)}
and Swedish Research Council (VR)}

\address{$^{1}$
Department of Mathematical Sciences,
Chalmers University of Technology and  G\"oteborg University,
SE--412 96, G\"oteborg, Sweden
}
\email{mohammad@chalmers.se}

\author[C.~Standar]{C. Standar$^{1, \dag}$}
\thanks{$^{\dag}$The corresponding author}
\email{standarc@chalmers.se}





\begin{abstract}
We consider the study of a numerical scheme 
for an initial- and Dirichlet boundary- value problem for a nonlinear 
Schr\"odinger equation. We approximate the solution 
using a, local (non-uniform) two level scheme 
 in time (see  C. Besse \cite{Besse:98_4} and \cite{Besse:2004_4}) 
combined with, an optimal, finite element strategy for 
the discretization in the spatial variable based on 
studies outlined as, e.g. 
in \cite{Akrivis_etal:91_4} and \cite{Lopez_Sanz-Serna:91_4}. 
For the proposed fully discrete scheme,
we show convergence both in $L_2 $ and $H^1$ norms.  
\end{abstract}

\maketitle

\section{Introduction \label{introduction}}
Let $T>0$ be a final time and 
$D\subset {\mathbb R}^d, d=1,2,3$, an arbitrary, convex and 
simply connected spatial domain. We consider the following 
initial- and Dirichlet boundary- value problem for a nonlinear 
Schr\"odinger equation: find a function 
$u:[0,T]\times \bar{D}\to{\mathbb C}$ satisfying 
\begin{eqnarray}\label{Maineq_4} 
u_t=i\Delta u+i f(\abs{u}^2) u + g(t, x), \quad 
\forall (t,x)\in (0, T]\times D \label{Shor1_4}\\
u(t,\cdot)|_{\partial D}=0, \qquad \forall t\in (0, T]\label{Shor1B_4}\\
u(0,x)=u_0,  \qquad \forall x\in D,\label{Shor1I_4} 
\end{eqnarray}
where $u_0 : \bar D \to {\mathbb C}$, $f\in C^3([0,\infty); {\mathbb R})$ and 
$g\in  C^3([0, T) \times {\mathbb R}^d; \C)$. For this problem we shall 
 study a fully discrete, optimal, space-time numerical scheme based on a 
spatial discretization strategy based on \cite{Akrivis_etal:91_4} combined 
with a two-level (half-step) Crank-Nicolson and Backward Euler temporal 
discretizations. 
(With $f, \, g: {\mathbb C}\to {\mathbb C}$ being locally 
Lipschitz, well-posedness of the problem with a sufficiently smooth 
solution requires 
further smoothness and compatibility assumptions, which we shall 
consider in the semi-discrete approximation.)

The nonlinear Schr\"odinger (NLS) equation 
 is modeling several physical phenomena describing, e.g.,  quantum effects, 
with a solution that describes molecular, atomic, subatomic as well as  
 macroscopic systems.  In particular the cubic NLS (when $ f(x) = \lambda x $ for real number $ \lambda $) 
is of vital interest in applications such as nonlinear optics, 
oceanography and plasma physics. 
For a survey of significant mathematical results 
on Schr\"odinger equation we refer to an early work by 
Strauss \cite{Strauss:79_4}. Previous studies related to this
work can be found in, e.g., Akrivis et al.\ in 
\cite{Akrivis_etal:2003_4} and
\cite{Karakashian_etal:93_4}. The study in 
\cite{Akrivis_etal:2003_4} concerns an initial value problem for a radially 
symmetric nonlinear Schr\"odinger equation in 2 and 3 spatial dimensions 
discretized by a standard Galerkin combined with a 
Crank-Nicolson type time-stepping. While in 
\cite{Karakashian_etal:93_4} the authors 
consider an implicit Runge-Kutta temporal 
approximation combined with the Galerkin method for the spatial domain. 
In both studies the spatial scheme is on the background and the focus 
of analysis is on the time discretization.  
We shall give a brief approach to a more standard spatial discretization. 
In this part we relay on investigations by Akrivis and co-authors 
\cite{Akrivis_etal:91_4}, where also a second order accurate 
temporal discretizations based on a Crank-Nicolson scheme is studied.
There is a more abstract approach by Tourigny \cite{Yves:91_4} that relies on 
the nonlinear stability theory developed in 
\cite{Lopez_Sanz-Serna:91_4}. 
In  \cite{Yves:91_4}, a pointwise error bound is established and  
$H^1$ optimal estimates are derived for 
both backward-Euler and Crank-Nicolson, temporal, schemes. 
Somewhat more elaborate studies employing the discontinuous Galerkin (DG)
strategy for spatial discretization are given in 
 \cite{Shu:2005_4} and
\cite{Asadzadeh:2013_4},
where the DG method
for the coupled NLS equations are considered. More specifically, 
while \cite{Shu:2005_4} concerns the $ L_2 $-stability and implementations,
the work in \cite{Asadzadeh:2013_4} is devoted to
multiscale variational approach 
for the space-time discretization of a coupled
NLS equation and corresponding implementations. A related study,   
with a finite difference approach, is given 
by Akrivis and co-authors in \cite{Akrivis_etal:2001_4} for the
linear Schr\"odinger-type equation. 

In this paper we extend the uniform time scheme studied by 
Zouraris in \cite{Zouraris:2001_4} to the case of  
a two-time-level, non-uniform,  linearly implicit finite element scheme. 
The finite element approach for the spatial 
discretization is widely studied in full details inheriting
some crucial results from the nonlinear heat equation,
therefore, as mentioned above, the finite element discretization
will appear as a background scheme with its crucial results presented in 
overview form. Hence, although fully discrete, most of the new 
contribution in here concerns 
temporal approximation.  In this setting, 
assuming a regularity of the exact solution of 
$u\in H^2((0,t]; H^r(D))$,  
we prove convergence rate of order 
${\mathcal O}(k^2+h^{r})$ in the $L_2(D)$-norm and a gradient estimate
 with accuracy  
${\mathcal O}(k + h^{r-1})$ in the $H^1_0(D)$-norm.
The $ L_2 $-estimate in here is optimal. As for the gradient
estimate, comparing with the theoretical result by Tourigny
\cite{Yves:91_4}, our gradient estimate is sharp as well.
Whereas compared to Zouraris \cite{Zouraris:2001_4},
where the spatial gradient estimate is of order
$ \mathcal{O} (k^2 + h^{r - 1}) $, due to the fact
that there is no time derivate involved, our result
is suboptimal.

An outline of this paper is as follows. In Section 2 we introduce some notation 
and preliminaries necessary in the analysis. In Section 3 we 
introduce two related spatial discretization strategies, study 
the convergence of the simplest one and derive the optimal semi-discrete 
error estimates. The results in here are of overview nature and are 
for the sake of completeness. Section 4 is devoted to the study of a two-level 
non-uniform grid time discretization
of a (background) Galerkin finite element solution   
obtained in Section 3. In this section we also
include the consistency of the temporal scheme.
The convergence analysis is singled out and presented in the 
concluding Section 5. Finally in Section 6 we give a conclusion 
of the results of the paper. 
Throughout the paper $C$ will denote a 
generic constant that might be different at  different appearances and is 
independent of all involved parameters and functions unless otherwise 
explicitly stated.

\section{Notation and preliminaries} 
We employ the $L_2(D)$-based, complex inner product and the  bilinear form
$$
(u,v)=\int_Du(x)\overline{v(x)}\, dx,\quad\mbox{and}\quad 
a(u,v)=\int_D\nabla u(x)\overline{\nabla v(x)}\, dx, 
$$ 
respectively.  
For a  {\sl multi-index} 
$\alpha=(\alpha_1, \alpha_2\ldots,\alpha_d)$,
$\alpha_i\ge 0$, and with $\abs{\alpha}=\alpha_1+ \alpha_2+\ldots+\alpha_d$   
we recall the standard Sobolev space  
$$ 
H^s (D):=\{v: {\mathcal D}^\alpha v\in L_2(D); \abs{\alpha}\le s\}, 
$$
of 
all complex-valued functions in $L_2(D)$, also 
having all their distribution derivatives of order $\le s$ in 
$L_2(D)$. $ H^s (D) $ is associated
with the norm and seminorm
$$
\norm{v}_s=\norm{v}_{H^s(D)} 
=\Big(\sum_{ \abs{\alpha}\le s}
\norm{{\mathcal D}^\alpha v}^2_{L_2(D)}\Big)^{1/2}, \quad \mbox{and}\quad 
\abs{v}_s
=\Big(\sum_{ \abs{\alpha}=s}
\norm{{\mathcal D}^\alpha v}^2_{L_2(D)}\Big)^{1/2},  
$$
respectively. Hence $L_2=H^0$ and we denote 
$\norm{v}:=\norm{v}_{L_2}=\norm{v}_0$. 
Further we define 
$$
H^1_0(D):=\{v\in H^1 (D); v=0 \mbox{ on } \partial D\}.
$$
We shall also use the space of continuously differentiable functions 
$C^m(D)$ on $D$, consisting of the space of all complex valued functions 
$v$ with all their partial derivatives  ${\mathcal D}^\alpha v$ of order 
$ \abs{\alpha}\le m$ being continuous in $D$. Further, 
$C^m(\overline D)$ presents the space of functions $v\in  C^m(D)$ 
for which ${\mathcal D}^\alpha v$ is bounded and uniformly continuous 
in $D$ for $ \abs{\alpha}\le m$. The norm in $C^m(\overline D)$ is defined as 
$$
\norm{v}_{C^m(\overline D)}=\max_{\abs{\alpha}\le m}\sup_{x\in \bar D}
\abs{{\mathcal D}^\alpha v(x)}.
$$ 
In this setting it is easy to verify that the solution $u(t):=u(t, \cdot)$ 
satisfies the boundedness relation
in the sense that  
\begin{equation}\label{invariant_4}
\norm{u(t)}\leq \norm{u_0} + \int_0^t \norm{g(\tau, \cdot)} \, d\tau,\qquad 0\le t\le T. 
\end{equation}

Frequently an abstract and
extended version of the Sobolev spaces to time dependent functions
appear in our reference literature; see, e.g., 
\cite{Evans:2002_4}. Below we include a brief notation:
For a Banach space $X$ with the norm 
$\norm{\cdot}_X$  the function space 
$L_p (0, T; X)$ consists of all strongly measurable functions 
$\bm{u}:[0, T]\to X$ with 

\begin{equation*}
\norm{\bm{u}}_{L_p (0,T; X)} :=
\left\{  
\begin{array}{ll} \displaystyle 
\Big(\int_0^T\norm{\bm{u}(t)}_X^p \, dt\Big)^{1/p}, 
\qquad & 1\le p<\infty, \\
\mbox{ ess } \sup_{0\le t\le T} \norm{\bm{u}(t)}_X,
& p=\infty. 
\end{array}
\right.
\end{equation*}
Then the Sobolev space $W^{s,p}(0, T; X)$ is defined by boundedness of 
the norms of its elements, viz. 
\begin{equation*}
\norm{\bm{u}}_{W^{s,p}(0,T; X)}=
\left\{  
\begin{array}{ll} \displaystyle 
\Big(\sum_{m=0}^s\int_0^T
\norm{\frac{\partial^m \bm{u}}{\partial t^m}}_X^p\, dt\Big)^{1/p}, 
\qquad & 1\le p<\infty, \\
\max_{0\le m\le s} \mbox{ ess } \sup_{t\in [0,T]} 
\norm{\frac{\partial^m\bm{u}}{\partial t^m}(t, \cdot)}_X, 
& p=\infty. 
\end{array}
\right.
\end{equation*}
Obviously $0$ and $T$ can be replaced by 
 any $t_1$ and $t_2$ with $0\le t_1<t_2$ 
and hence $[0, T]$  by $[t_1, t_2]$.

\section{The spatial discretization Scheme}  
The spatial discretization for the Schr\"odinger equation is by now 
a standard procedure, considered by several authors.
At first glance it can be viewed as an extension
of the results for heat equation. This however is
not a straight-forward strategy, due to
the complex terms in the NLS.
Nevertheless, below, for the sake of completeness,
we introduce two equivalent spatial 
discretization strategies outlined in \cite{Yves:91_4} and 
\cite{Akrivis_etal:91_4}, respectively:  

{\bf SI}. 
Let $N$ be a positive integer, and, for $n=1,\ldots, N$, 
let  $\{S_h^n\}_{h\in (0,1)}\subset H_0^1(D)$ 
be a family of finite dimensional subspaces. 
Consider a partition of the time interval $[0, T]$ 
into not necessarily uniform subintervals $I_n:=(t_{n-1}, t_{n})$ and let 
$k_n:=\abs{I_n}=t_n-t_{n-1}$ be the length of $I_n$. Denoting by 
$u_h^n\in S_h^n$ an approximation of 
$u(t_n)$, we construct a vector 
$u_h=(u_h^0, u_h^1, u_h^2, \ldots, u_h^n) \in X_h$ by solving a discretized 
problem of the form 
\begin{equation*}\label{SI1}
\bm{ \Phi}_h (u_h)=0.
\end{equation*} 
Here $\bm{\Phi}_h:X_h\to Y_h$ is a nonlinear mapping with 
$$
\left\{
\begin{aligned}
X_h & = S_h^0\times S_h^1 \times \cdots\times S_h^N\\
Y_h & = (S_h^0)^*\times (S_h^1)^* \times \cdots\times (S_h^N)^*.
\end{aligned}
\right. 
$$ 
where $(S_h^n)^*$ is the dual space of 
$(S_h^n, \norm{\cdot}),\, n=0,\ldots , N$ 
equipped with the norm 
$$
\norm{\cdot}_*=
\sup_{\varphi\in S_h^n}\frac{\abs{\langle \cdot, \varphi \rangle}}{\norm{\varphi}}.
$$
Note that we have chosen $N+1$ different finite element spaces, 
with $S_h^n$ corresponding to discrete time level $t_n,\, n=0,\ldots, N$. 
Detailed stability and convergence analysis in this regi are given 
by Lopez-Marcos and Sanz-Serna in \cite{Lopez_Sanz-Serna:91_4} 
(see also the analysis in \cite{Yves:91_4}). 

Our focus will be on a simpler strategy, based on an approach by 
\cite{Akrivis_etal:91_4} as outlined below:  
 
{\bf SII}. 
At each time level,  
let $\{S_h\}_{h\in (0,1)}$ be a family of finite dimensional subspaces of 
${\mathcal H}_0:=H_0^1(\Omega)\cap C(\overline D) $ 
satisfying the approximation property
\begin{equation}\label{Approx1_4}
\inf_{\chi\in S_h}\Big\{ \norm{v-\chi}+h \norm{v-\chi}_1\Big\}
\le C h^s\norm{v}_s,\quad \forall v\in H^s(D)\cap H^1_0(D), 
\quad 2\le s\le r, 
\end{equation}
for all  $h\in (0,1)$, where $s$ is an integer. Here for  
a quasi-uniform 
family of partitions of $D$,  
$\{S_h\}$ is the set of all continuous functions with their  
real and imaginary parts being piecewise polynomials of degree $r-1$ on 
$D$, where 
$r\ge 2$. Then, for $\varphi\in S_h$, we have the following inverse 
inequalities due to \cite{Suli:88_4}: 
\begin{equation}\label{Invers1_4}
\norm{\varphi}_{L_\infty(D)}\le C
\left\{
\begin{aligned}
&h^{-\frac d2}\norm{\varphi},\qquad & d=1,2,3, \\
&h^{1-\frac d2}(\log(1/h))^{1-1/d}\norm{\nabla\varphi}, & d=2, 3, 
\end{aligned}
\right.
\quad \forall \varphi\in S_h.
\end{equation}
Now using \eqref{Approx1_4}, \eqref{Invers1_4} and assuming 
existence of certain operator bounds,
one can deduce that (see \cite{Akrivis_etal:91_4})
\begin{equation}\label{Invers2_4} 
\lim_{h\to 0} \, \sup_{0\le t\le T} \, \inf_{\varphi\in S_h} 
\{\abs{u(t)-\varphi}_\infty+h^{-\frac d2} \norm{u(t)-\varphi}\}=0.
\end{equation}
To proceed, for $h\in (0,1)$, we define the discrete Laplacian operator 
$\Delta_h:S_h\to S_h$, as 
$$
(\Delta_h \varphi, \chi)=-(\nabla\varphi, \nabla\chi), \quad \forall 
\varphi, \, \chi\in S_h,
$$
and an elliptic projection operator 
$R_h:H^1(D)\to S_h$ by 
$$
(\nabla R_h v, \nabla\chi)= (\nabla v, \nabla\chi), \quad 
\forall v\in H^1(D), \quad \forall\chi\in S_h.  
$$
Then (see, e.g., \cite{Vidar_4} and \cite{Al_Lars:82_4}) 
$R_h$ satisfies the approximation properties
\eqref{Approx1_4}, viz. 
\begin{equation}\label{Ellip1_4}
\norm{R_h v-v}+h \norm{ R_h v- v}_1
\le C_R h^s\norm{v}_{s},\quad \forall v\in H^s(D)\cap H^1_0(D), 
\quad 2\le s\le r, 
\end{equation}
and 
\begin{equation*}\label{Ellip1A_4}
\norm{R_h v(t)-v(t)}_{L_\infty(D)} 
\le C_R (\log(1/h)^{\bar s}h^s\norm{v}_{W^{s,\infty}(D)},
\end{equation*}
for all $ h\in (0,1)$, where $\bar s=1$ if $s=2$ and $d\ge 2$, 
and zero otherwise. Further 
\begin{equation*}\label{Ellip2_4}
\norm{\nabla R_h v}\le \norm{\nabla v},\quad  \forall v\in H^1(D), 
\quad \forall h\in (0,1).
\end{equation*}
Now a, time continuous, 
variational formulation for the problem \eqref{Shor1_4} reads as follows: 
Find $u\in H^1_0(D)$, such that 
\begin{equation*}\label{Varform1_4}
\left\{
\begin{array}{l}
(u_t, \chi)+i(\nabla u, \nabla \chi)-i(f(\abs{u}^2)u , \chi)=(g, \chi), 
\qquad \forall \chi \in H^1_0(D), \\
u(0,x)=u_0(x).
\end{array}
\right.
\end{equation*}
The corresponding finite element problem is
formulated as finding $u_h\in S_h$, such that 
\begin{equation}\label{FEM1_4}
\left\{
\begin{array}{l}
(u_{h,t}, \chi)+i(\nabla u_h, \nabla \chi)-i(f(\abs{u_h}^2)u_h , \chi)
=(g, \chi), 
\qquad \forall \chi \in S_h \\
u_h(0,x)=u_{0,h}(x),
\end{array}
\right.
\end{equation}
where $u_{0,h}$ is an approximation of $u_0$ in $S_h$.

\subsection{Spatial/semidiscrete error estimate} 
As we mentioned in the introduction, the finite element 
schemes for the spatial discretization of the Schr\"odinger equation 
\eqref{Maineq_4} are 
fully considered in the literature. This section is a 
brief review of two equivalent spatial 
discretization strategies 
 adequate in our fully discrete study.  

As a crucial property of the finite element scheme \eqref{FEM1_4}, 
in analogy with \eqref{invariant_4},   
we can easily verify that the $L_2$-norm of the semidiscrete solution 
$u_h(t):=u_h(\cdot, t)$ is bounded in the following sense
\begin{equation}\label{FEMinvariant_4}
\norm{u_h(t)} \leq \norm{u_{0,h}} + \int_0^t \norm{g(\tau, \cdot)} \, d\tau, \qquad 0\le t\le T.
\end{equation}
Then, the existence of a unique solution $u_h$ for 
\eqref{FEM1_4} would follow recalling that $f$ is locally Lipschitz
together with the relations \eqref{Invers1_4} and \eqref{FEMinvariant_4}. 

The convergence estimate for this semidiscrete problem is derived in 
 \cite{Akrivis_etal:91_4}. 
Below, for the sake of completeness, 
we outline a concise approach to their proof:

\begin{theorem} Assume that $u\in H^s(D)$. Then, 
  the finite element solution $u_h(t)$ for \eqref{FEM1_4} inherits the 
convergence rate for an appropriately chosen approximation $u_{0,h}$ 
for $u_0$ and 
yields the optimal convergence rate, viz. 
\begin{equation*}\label{FEMconvergent1_4}
\norm{u_0-u_{0,h}}\le Ch^s\Longrightarrow 
\max_{0\le t\le T}\norm{u(t)-u_{h}(t)}\le Ch^s\norm{u}_s.
\end{equation*}
\end{theorem}
\begin {proof}
Given $\varepsilon>0$, let 
$
M_\varepsilon=\{z\in\mathbb C: \exists (x,t)\in \overline D\times [0,T]\,\, 
\abs{z-u(x,t)}<\varepsilon\}, 
$ 
and define a globally Lipschitz function 
$f_\varepsilon:{\mathbb C}\to {\mathbb C}$, so that 
$f_\varepsilon(z)=f(z)$ for $z\in M_\varepsilon$. 
Now we define an auxiliary function $w_h:[0,t]\to S_h$ 
as the unique solution of 
\begin{equation}\label{AuxiliaryFEM1_4}
\left\{
\begin{array}{l}
(w_{h,t}, \chi)+i(\nabla w_h, \nabla \chi)-i(f_\varepsilon(\abs{w_h}^2)w_h , \chi)
=(g, \chi), 
\qquad \forall \chi \in S_h \\
w_h(0,x)=u_{0,h}(x), 
\end{array}
\right.
\end{equation}
The proof is now based on first establishing the auxiliary estimate 
\begin{equation}\label{AuxiliaryFEM2_4}
\max_{0\le t\le T}\norm{u(t)-w_{h}(t)}\le 
C\Big(\norm{u_0-u_{0,h}}+ h^s\norm{u}_s\Big), 
\end{equation}
and then justifying the fact that indeed, for sufficiently small $h$,  
$u_h$ and $w_h$ coincide. To show \eqref{AuxiliaryFEM2_4} we shall use the 
split $u-w_h=(u-R_hu)+(R_h u-w_h):=\rho+\theta$. Now, a combination of 
\eqref{Maineq_4}, \eqref{FEM1_4}, and \eqref{AuxiliaryFEM1_4} yields 
$$
(\theta_{t}, \chi)+i(\nabla \theta, \nabla \chi)
=-(\rho_{t}, \chi)-
i(f_\varepsilon(\abs{w_h}^2)w_h- f_\varepsilon(\abs{u}^2)u , \chi),\,\,\, 0\le t\le T,
$$
where we used the fact that $f_\varepsilon$ coincides with $f$ in $M_\varepsilon$. 
Next, we take $\chi=\theta$ and consider the real part to get 
$$
\frac 12\frac{d}{dt}\norm{\theta(t)}^2\le 
\Big(\norm{f_\varepsilon(\abs{w_h}^2)w_h- f_\varepsilon(\abs{u}^2)u}+
\norm{\rho_{t}}\Big)\norm{\theta(t)}. 
$$
Consequently 
\begin{equation*}\label{AuxiliaryFEM3_4}
\begin{split}
\frac{d}{dt}\norm{\theta(t)}\le & L\norm{w_h-u}\norm{w_h+u}_{\infty}\norm{w_h}+
\norm{w_h-u}\norm{f_\varepsilon(\abs{u}^2)}+\norm{\rho_{t}}\\
&\le 
C(\norm{\theta(t)}+\norm{\rho(t)}+\norm{\rho_{t}(t)}+ \norm{u_0}
 +\norm{u_{0,h}}).
\end{split}
\end{equation*}
where we have used the Lipschitz continuity of $f$ with Lipschitz constant $L$.
Here we assumed that $ \max ( \norm{u}_\infty, \norm{w_h}_\infty) < C(t) $,
which can be motivated by the stability estimates 
\eqref{invariant_4} and \eqref{FEMinvariant_4} (the latter for $w_h$ replacing 
$u_h$). Now by Gr\"onwall's lemma, the property of the 
elliptic operator $R_h$ and \eqref{Ellip1_4} we deduce the desired estimate 
\eqref{AuxiliaryFEM2_4}. Evidently combining the assumption of the theorem: 
$\norm{u_0-u_{0,h}} \le Ch^s $ and \eqref{AuxiliaryFEM2_4} yields 
\begin{equation}\label{AuxiliaryFEM4_4}
\max_{0\le t\le T}\norm{u(t)-w_{h}(t)}\le Ch^s.
\end{equation}
Further, by \eqref{Invers1_4}, \eqref{Ellip1_4} and \eqref{AuxiliaryFEM4_4}, we have for $0\le t\le T$ and $\chi\in S_h$, 
\begin{equation*}\label{AuxiliaryFEM5_4}
\begin{split}
\norm{u(t)-w_h(t)}_\infty\le & 
 \norm{u(t)-\chi}_\infty+\norm{\chi-R_hu(t)}_\infty+\norm{R_hu(t)-w_h(t)}_\infty\\
\le & \norm{u(t)-\chi}_\infty+Ch^{-\frac d2}(\norm{u(t)-\chi}+\norm{\rho(t)}+
\norm{\theta(t)})\\
\le & \norm{u(t)-\chi}_\infty+Ch^{-\frac d2}\norm{u(t)-\chi}+Ch^{s-\frac 12}. 
\end{split}
\end{equation*}
Now, recalling \eqref{Invers2_4}, we deduce that:
$$
\exists h_0>0: \forall h\le h_0, w_h(x,t)\in M_\varepsilon, \quad 
(x,t)\in \overline D\times [0,T].
$$
For such $h$, $u_h=w_h$ and the proof follows from \eqref{AuxiliaryFEM4_4}.
\end{proof}
\section{A time discretization scheme} 
Let $N\in {\mathbb N}$ and $\{t_n\}_{n=0}^N$ be the nodes of a non-uniform partition 
of the time interval $[0, T]$, i.e. 
$t_n<t_{n+1}$ for $n=0, 1, \ldots, N-1$ , $t_0=0$ and $t_N=T$. 
Then, we set $k_n:= t_n-t_{n-1}$ for $n=0, 1, \ldots, N$ and consider the 
following two time-step numerical scheme: 

{\bf Step 1.} Set 
$$
U_h^0=u_{0,h},
$$
where 
$ u_{0,h} = R_h u_0 $.

 {\bf Step 2.} For $n=1,2, \ldots, N$, first find $U_h^{n-\frac 12}\in S_h$ 
such that 
\begin{equation*}
\begin{split}
\left( \frac{U_h^{n-\frac 12}-U_h^{n-1}}{k_n/2}, \chi \right) =&
i\left( \frac{\nabla U_h^{n-\frac 12}+ \nabla U_h^{n-1}}{2}, \nabla \chi \right) \\
 &+ i  \left( f(\abs{U_n^{n-1}}^2)\frac{U_h^{n-\frac 12}+U_h^{n-1}}{2}, \chi \right)
+ \left( g( t_{n-1}, \cdot), \chi \right),
\end{split}
\end{equation*}
for all $ \chi \in S_h $ and then find $U_h^{n}\in S_h$ such that 
\begin{equation*}
\begin{split}
\left( \frac{U_h^{n}-U_h^{n-1}}{k_n}, \chi \right)=&
i \left( \frac{\nabla U_h^{n}+ \nabla U_h^{n-1}}{2}, \nabla \chi \right) \\
&+ i \left( f(\abs{U_h^{n-\frac 12}}^2)\frac{U_h^{n}+U_h^{n-1}}{2}, \chi \right)
+ \left( g(t_{n-\frac 12}, \cdot), \chi \right),
\end{split}
\end{equation*}
for all $ \chi \in S_h $.
 
\subsection{Consistency}
Below we show the consistency of the scheme defined in step 2. 
To this approach,  
for $n=1,2,\ldots, N$, we define $r^{n-\frac 12}$ and $r^n$ by 
\begin{equation}\label{rn0.5_4}
\begin{split}
\frac{u^{n-\frac 12}-u^{n-1}}{k_n/2} =&
i\Delta \Big(\frac{u^{n-\frac 12}+u^{n-1}}{2}\Big)
+i f(\abs{u^{n-1}}^2)\frac{u^{n-\frac 12}+u^{n-1}}{2} \\
&+g(t_{n-1}, \cdot)+r^{n-\frac 12}
\end{split}
\end{equation}
and 
\begin{equation}\label{rn1_4}
\begin{split}
\frac{u^{n}-u^{n-1}}{k_n}=&
i\Delta \Big(\frac{u^{n}+u^{n-1}}{2}\Big)
+i f(\abs{u^{n-\frac 12}}^2)\frac{u^{n}+u^{n-1}}{2}\\ 
&+g(t_{n-\frac 12}, \cdot)+r^n, 
\end{split}
\end{equation}
respectively, where $u^n=u(t_n, \cdot)$ for $n=0, 1, \ldots, N$. 
Then, we have the following estimates for $r^{n-\frac 12}$ and $r^n$: 
\begin{proposition} Assume that there is a constant $C_1$ such that 
\begin{equation}\label{rn0.5A_4}
\max\Big( \norm{\partial_t u}_\infty, \,
\norm{\partial_t^2 u}_\infty,\, \norm{\Delta\partial_t u}_\infty \Big)< C_1, 
\end{equation}
then  
\begin{equation*}
\norm{r^{n-\frac 12}}\le C k_n. 
\end{equation*}
Further if, in addition to \eqref{rn0.5A_4},  we have that there is a constant 
$C_2$ such that 
\begin{equation}\label{rn1A_4}
\max\Big( \norm{\partial_t^3 u}_\infty, \,
\, \norm{\Delta\partial_t^2 u}_\infty \, \Big)< C_2, 
\end{equation}
then 
\begin{equation*}
\norm{r^{n}}\le C k_n^2.
\end{equation*}
\end{proposition}
\begin{proof}
We start proving the second assertion.
Subtracting the Schr\"odinger equation at time $t=t_{n-\frac 12}$ from 
the equation \eqref{rn1_4} 
yields 
\begin{equation*}
\begin{split}
r^n =& \frac{u^{n}-u^{n-1}-k_n u_t^{n-\frac 12}}{k_n} - 
i\Delta\Big(\frac{u^{n}+u^{n-1}-2u^{n-\frac 12}}{2}\Big)\\
&-i f(\abs{u^{n-\frac 12}}^2)
\Big(\frac{u^{n}+u^{n-1}-2u^{n-\frac 12}}{2}\Big)  
:=J_1-J_2-J_3.
\end{split}
\end{equation*}
We estimate each 
$\norm{J_i}$, $i=1,2,3$, separately.  
To this end we Taylor expand $u^n$ and $u^{n-1}$ about $t_{n-\frac 12}$ 
of degree 2 for $J_1$ and of degree 1 for $J_2$ and $J_3$. As for $J_1$, 
by cancellations in Taylor expansions we end up with 
\begin{equation*}\label{rnJ1A}
 J_1=\frac{1}{6 k_n}\int_{t_{n-\frac 12}}^{t_n} 
\partial_t^3 u(t, \, \cdot)(t_n-t)^2\, dt - 
\frac{1}{6 k_n}\int_{t_{n-1}}^{t_{n-\frac 12}} 
\partial_t^3 u(t, \, \cdot)(t_{n-1}-t)^2\, dt. 
\end{equation*}
Now since $(t_n-t)^2\le k_n^2/4$ on the interval $(t_{n-\frac 12}, t_n)$, 
likewise $(t_{n-1}-t)^2\le k_n^2/4$ on the interval 
$(t_{n-1}, t_{n-\frac 12})$, we have using \eqref{rn1A_4} that 
\begin{equation*}\label{rnJ1B}
\norm{J_1}\le \frac{k_n}{24}\int_{t_{n-1}}^{t_n} 
\norm{\partial_t^3u(t,\, \cdot)}\, dt \le C k_n^2. 
\end{equation*}
Similarly 
\begin{equation*}\label{rnJ2A}
 J_2=\frac{i\Delta}{4}\Big(\int_{t_{n-\frac 12}}^{t_n} 
\partial_t^2 u(t, \, \cdot)(t_n-t)\, dt - 
\int_{t_{n-1}}^{t_{n-\frac 12}} 
\partial_t^2 u(t, \, \cdot)(t_{n-1}-t)\, dt\Big), 
\end{equation*}
where using 
$(t_n-t)\le k_n/2$ on the interval $(t_{n-\frac 12}, t_n)$ and 
 $(t_{n-1}-t)\le k_n/2$ on $(t_{n-1}, t_{n-\frac 12})$
together with \eqref{rn1A_4}, we obtain
\begin{equation*}\label{rnJ2B}
\norm{J_2}\le \frac{k_n}{8}\int_{t_{n-1}}^{t_n} 
\norm{\Delta\partial_t^2u(t,\, \cdot)}\, dt \le C k_n^2. 
\end{equation*}
As for $J_3$ we have 
\begin{equation*}\label{rnJ3A}
J_3= \frac{i}{4} f\Big(\abs{u^{n-\frac 12}}^2\Big)
\Big(\int_{t_{n-\frac 12}}^{t_n} 
\partial_t^2 u(t, \, \cdot)(t_n-t)\, dt - 
\int_{t_{n-1}}^{t_{n-\frac 12}} 
\partial_t^2 u(t, \, \cdot)(t_{n-1}-t)\, dt\Big), 
\end{equation*}
and hence, using \eqref{rn0.5A_4} combined with the
assumption on $ f $, we have 
\begin{equation*}\label{rnJ3B}
\norm{J_3}\le Ck_n 
\int_{t_{n-1}}^{t_n} \norm{\partial_t^2 u(t,\, \cdot)}\, dt \le C k_n^2. 
\end{equation*}
Consequently, we have proved the second assertion that  
\begin{equation*}
\norm{r^{n}}\le C k_n^2.
\end{equation*}

The proof of the first assertion, basically, follows
a similar path, however for the sake of completeness
we include it as well.
This time we subtract the Schr\"odinger equation at
the time level $t=t_{n-1}$ from the equation \eqref{rn0.5_4}, which yields 
\begin{equation*}
\begin{split}
r^{n-1/2} =& \frac{u^{n-1/2}-u^{n-1}-\frac {k_n}2 u_t^{n-1}}{k_n/2} - 
i\Delta\Big(\frac{u^{n-1/2}-u^{n-1}}{2}\Big)\\
&-i f(\abs{u^{n-1}}^2)
\Big(\frac{u^{n-1/2}-u^{n-1}}{2}\Big)  
=S_1-S_2-S_3.
\end{split}
\end{equation*}
Below we estimate the norms 
$\norm{S_i}, \, i=1,2,3$, using 
Taylor expansion of $u^{n-1/2}$ about $t_{n-1}$, of order 1 for $S_1$ 
and order $0$ for $S_2$ and $S_3$. As for $S_1$ we have 
\begin{equation*}\label{rnS1A}
S_1 = \frac 1{k_n}\int_{t_{n-1}}^{t_{n-1/2}}\partial_t^2 u(t,\, \cdot) 
(t_{n-1/2}-t)\, dt. 
\end{equation*}
Thus using  \eqref{rn0.5A_4} we deduce that 
\begin{equation}\label{rnS1B_4}
\norm{S_1}\le \frac 12 \int_{t_{n-1}}^{t_{n-1/2}}
\norm{\partial_t^2 u(t,\, \cdot)}\, dt \le C k_n. 
\end{equation}
The $S_2$ term is then 
\begin{equation*}\label{rnS2A_4}
S_2 = \frac {i\Delta}2\int_{t_{n-1}}^{t_{n-1/2}}\partial_t u(t,\, \cdot) 
\, dt,
\end{equation*}
which, by \eqref{rn0.5A_4}, can be estimated as 
\begin{equation}\label{rnS2B_4}
\norm{S_2}\le \frac 12 \int_{t_{n-1}}^{t_{n-1/2}}
\norm{\Delta\partial_t u(t,\, \cdot)}\, dt \le C k_n. 
\end{equation}
Finally we have 
\begin{equation*}\label{rnS3A}
S_3 = \frac {i}2 f\Big(\abs{u^{n-1}}^2\Big)
\int_{t_{n-1}}^{t_{n-1/2}}\partial_t u(t,\, \cdot) 
\, dt,
\end{equation*}
Once again using \eqref{rn0.5A_4}, we can estimate the $S_3$ term as 
\begin{equation}\label{rnS3B_4}
\norm{S_3}\le \frac{1}{2} \norm{ f\Big(\abs{u^{n-1}}^2\Big)}  
\int_{t_{n-1}}^{t_{n-1/2}}
\norm{\partial_t u(t,\, \cdot)}\, dt \le C k_n. 
\end{equation}
Summing up, the estimates \eqref{rnS1B_4}, \eqref{rnS2B_4} and 
\eqref{rnS3B_4} yields the first assertion of the theorem 
 and gives the estimate for $r^{n-1/2}$, 
and the proof is complete.  

\end{proof}

\section{Convergence Analysis} 

In this part we rely on a result by Zouraris \cite{Zouraris:2001_4}, viz. 

\begin{lemma} \label{ZourarisLemma2.1_4}
Let $ u_1, u_2 \in C (\bar{D}) $ and $ g \in C^1 ( [0, \infty ); \R) $, then 
we have
\begin{equation*}
\norm{g \left( \abs{u_1}^2 \right) - g \left( \abs{u_2}^2 \right)}
\leq \sup_{x \in I(u_1, u_2)} \abs{g' (x)} ( \norm{u_1}_\infty 
+ \norm{u_2}_\infty ) \norm{u_1 - u_2}
\end{equation*}
with $ I(u_1, u_2) := 
[0, \max \{ \norm{u_1}_\infty^2, \norm{u_2}_\infty^2 \} ] $.
\end{lemma}

Now we are ready to formulate the main result of this paper: 

\begin{theorem}
Let $e^n:=U_h^n-u^n$ be the error at the time level $t=t_n$. 
Assume that $u $ satisfies the conditions \eqref{rn0.5A_4} and \eqref{rn1A_4}.
Then there is a constant $C$ such that 
$$
\norm{e^n}\le C(k^2+h^r), 
$$
and 
$$
\norm{\nabla e^n}\le C(k + h^{r-1}),
$$
with $k:=\max_{1\le n\le N}k_n$.
\end{theorem} 
\begin{proof} 
We start proving the $L_2$ estimate first. To this end 
we rely on the classical approach and split the error $e^n$ 
into the Ritz projection error at the time level 
$n$ and the error between the fully approximate solution $U_h^n$ 
and the Ritz projection: 
\begin{equation*} \label{errorsplit}
e^n:= U_h^n-u^n= (U_h^n-R_h u^n)+ (R_h u^n- u^n)=:\theta^n+\rho^n. 
\end{equation*}
We invoke the standard estimate for the Ritz projection 
error $\rho^n$ from the theory and focus on the estimates for 
$\theta^n$. Note first that $ \theta^n $ satisfies the following
equation 
\begin{equation}\label{Theta1_4}
\begin{split}
\left( \frac{\theta^n-\theta^{n-1}}{k_n }, \chi \right) =& 
- i \left( \nabla \frac{\theta^n + \theta^{n - 1}}{2}, \nabla \chi \right)
- \left( \frac{\rho^n-\rho^{n-1}}{k_n }, \chi \right) \\
&+ i ( \omega^n, \chi) - ( r^n, \chi),
\end{split}
\end{equation}
with
\begin{equation*}
\omega^n = f\Big(\abs{U_h^{n-1/2}}^2\Big) \frac{U_h^n+U_h^{n-1}}{2}
-f\Big(\abs{u^{n-1/2}}^2\Big)\frac{u^n+u^{n-1}}{2}.
\end{equation*}
Choosing $ \chi = \theta^n+\theta^{n-1} $ leads to
\begin{equation*}\label{Theta2A}
\begin{split}
\norm{\theta^n}^2- \norm{\theta^{n-1}}^2=&
-\frac{ik_n}{2}\norm{\nabla(\theta^n+\theta^{n-1})}^2
- \left( \rho^n-\rho^{n-1}, \theta^n+\theta^{n-1} \right) \\
&+ i k_n (\omega^n,\theta^n+\theta^{n-1})
- k_n(r^n, \theta^n+\theta^{n-1}).  
\end{split}
\end{equation*}
Obviously the first term on the right-hand side above is purely imaginary. 
Taking the real part of the other terms on the right-hand side 
 we end up with 
\begin{equation}\label{Theta2_4}
\begin{split}
\norm{\theta^n}^2- \norm{\theta^{n-1}}^2=&
- \Re \left( \rho^n-\rho^{n-1}, \theta^n+ \theta^{n-1} \right) \\
&- k_n\Big[ \Im(\omega^n,\, \theta^n+ \theta^{n-1})
+\Re (r^n,\, \theta^n+ \theta^{n-1})  \Big]. 
\end{split}
\end{equation}
For the first term on the right-hand side we use the mean value
theorem together with \eqref{Ellip1_4} to get the estimate
\begin{equation}\label{Rho_4}
\abs{\left( \rho^n-\rho^{n-1}, \theta^n+ \theta^{n-1} \right)} \leq
C k_n h^r \left( \norm{\theta^n} + \norm{\theta^{n - 1}} \right).
\end{equation}
The estimate for $r^n$ was derived in the consistency section.
 It remains to estimate $\omega^n$. To do so we use the split: 
$$
\omega^n=\omega^n_1+\omega^n_2,
$$
where 
\begin{equation*}\label{Theta3}
\omega^n_1:= \Big(f\Big(\abs{U_h^{n-1/2}}^2\Big)-
f\Big(\abs{u^{n-1/2}}^2\Big)\Big)\frac{U_h^n+U_h^{n-1}}2,
\end{equation*}
and 
\begin{equation*}\label{Theta4}
\omega^n_2:= f\Big(\abs{u^{n-1/2}}^2\Big)\Big(\frac{e^n+e^{n-1}}2\Big).
\end{equation*}
To estimate $\omega^n_1$ and  $\omega^n_2$,
we assume that there exists a $\delta>0$ such that 
$$
\sup_{t\in [0,T]}\norm{u(t,\, \cdot)}_\infty+\max_n\norm{U_h^n}_\infty < \delta.
$$
Then by Lemma \ref{ZourarisLemma2.1_4} we have the following estimates:
\begin{equation}\label{Theta5_4}
\begin{split}
\norm{\omega_1^n} &\le \delta\sup_{x\in[0,\delta^2]}\abs{f^{\prime}(x)}
\cdot \Big(\norm{U_h^{n-1/2}}_\infty + \norm{u^{n-1/2}}_\infty\Big)
\norm{U_h^{n-1/2}-u^{n-1/2}}\\
& \le C\delta^2 \norm{e^{n-1/2}}.
\end{split}
\end{equation}
We have also 
\begin{equation}\label{Theta6_4}
\norm{\omega_2^n}\le \tilde C \Big( \norm{e^n}+  \norm{e^{n-1}}\Big),
\end{equation}
where
\[
\tilde C = \frac{1}{2} \sup_{x \in [0, \delta^2 ]} \abs{f(x)}.
\]
Inserting \eqref{Rho_4}, \eqref{Theta5_4}, \eqref{Theta6_4} and the estimate for $r^n$ into 
\eqref{Theta2_4}
we get 
\begin{equation*}
\begin{split}
\norm{\theta^n}^2- \norm{\theta^{n-1}}^2 \le &
\Big( C k_n \delta^2  \norm{e^{n-1/2}}+ \tilde{C} k_n \left( \norm{e^n}+ 
 \norm{e^{n-1}} \right)+ C k_n^3 + C k_n h^r \Big)\\
&\times \Big(\norm{\theta^n}+\norm{\theta^{n-1}}\Big).
\end{split}
\end{equation*}
Consequently 
\begin{equation}\label{Theta8_4}
\begin{split}
\norm{\theta^n}- \norm{\theta^{n-1}} & \le 
C k_n \delta^2  \norm{e^{n-1/2}}+ \tilde{C} k_n \left( \norm{e^n}+ 
 \norm{e^{n-1}}\right)+ C k_n(k_n^2+h^r) \\
&\le Ck_n \delta^2  \norm{\theta^{n-1/2}}+ \tilde{C} k_n \left( \norm{\theta^n}+ 
 \norm{\theta^{n-1}}\right)+ C k_n(k_n^2+h^r).
\end{split}
\end{equation}
Hence 
\begin{equation*}\label{Theta9}
\begin{split}
(1 - \tilde{C} k_n)\norm{\theta^n} \le 
C k_n \delta^2  \norm{\theta^{n-1/2}}+(1+ \tilde{C} k_n)
\norm{\theta^{n-1}} + C k_n(k_n^2+h^r).
\end{split}
\end{equation*}
A similar argument for $\theta^{n-1/2}$, using the estimate 
for $r^{n-1/2}$,  reads as follows: 
\begin{equation}\label{Theta10_4}
\begin{split}
(1-\frac{\tilde{C} k_n}2)\norm{\theta^{n-1/2}} \le 
\Big(1+\frac{\tilde{C} k_n}2 + C k_n \delta^2\Big)
\norm{\theta^{n-1}}
+ C k_n(k_n+h^r).
\end{split}
\end{equation}
To proceed we assume that 
$\tilde{C} k<1.$ Then a combination of \eqref{Theta8_4} and \eqref{Theta10_4} gives that 
\begin{equation*}\label{Theta11}
\norm{\theta^n}- \norm{\theta^{n-1}}  \le \tilde{C} k_n 
\norm{\theta^n} + C_n
\norm{\theta^{n-1}} + C k_n (k_n^2 + h^r),
\end{equation*}
where
\[
C_n :=\Big(
\frac{C k_n \delta^2(1+\frac{\tilde{C} k_n}2 + C k_n \delta^2)}{1-\frac{\tilde{C} k_n}2}+ \tilde{C} k_n
 \Big).
\]
Relabeling $n$ to $j$ and summing over $j=1,\ldots, n$, we thus obtain 
\begin{equation*}\label{Theta12}
\norm{\theta^n}- \norm{\theta^{0}}  \le 
\sum_{j=1}^n\Big(\norm{\theta^j}- \norm{\theta^{j-1}} \Big)
\le\sum_{j=1}^n \Big(\tilde{C} k_j \norm{\theta^j}+ C_j \norm{\theta^{j-1}}
+ C k_j (k^2+h^r)\Big), 
\end{equation*}
and hence 
\begin{equation*}\label{Theta13}
\begin{split}
(1- \tilde{C} k)\norm{\theta^n} \le C t_n (k^2+h^r)
+  (1 + C_1)\norm{\theta^0}
+\sum_{j=1}^{n-1}\Big[\tilde{C} k_j+ C_{j+1}\Big]
\norm{\theta^j}. 
\end{split} 
\end{equation*}
Then 
\begin{equation*}\label{Theta14}
\norm{\theta^n}\le \frac{C t_n (k^2+h^r )}{1- \tilde{C} k }+
\frac{1 + C_1}{1- \tilde{C} k }\norm{\theta^0}
+ \sum_{j=1}^{n-1}
\frac{\tilde{C} k_j + C_{j+1}}{1- \tilde{C} k}
\norm{\theta^{j}}, 
\end{equation*}
so that by the discrete Gr\"onwall inequality we get 
\begin{equation}\label{Theta15_4}
\norm{\theta^n}\le \frac{C t_n (k^2+h^r )}{1- \tilde{C} k }
\exp\Big[\frac{1+ C_1}{1- \tilde{C} k }
+\sum_{j=1}^{n-1}
\frac{\tilde{C} k_j+ C_{j+1}}{1- \tilde{C} k}
  \Big].
\end{equation}
Thus for the error $e^n$ we have the $L_2$ estimate 
\begin{equation*}\label{Theta16}
\norm{e^n}\le \norm{\theta^n}+\norm{\rho^n}\le C t_n (k^2+h^r).
\end{equation*}
Combining \eqref{Theta10_4} and \eqref{Theta15_4} we also have 
\begin{equation*}\label{Theta17}
\norm{\theta^{n-1/2}}\le C t_n (k^2+h^r), 
\end{equation*}
and therefore we have the same $L_2$ estimate for the error 
$e^{n-1/2}$ in the intermediate 
time level: 
\begin{equation*}\label{Theta18}
\norm{e^{n-1/2}}\le C t_n (k^2+h^r).
\end{equation*}
It remains to derive the $L_2$-estimate for the gradient of the error: 
$\norm{\nabla e^n}$. Even in here the estimate for $\norm{\nabla \rho^n}$ is 
an approximation theory result and we need to give an error bound for 
$\norm{\nabla \theta^n}$. To this end we choose $ \chi = \theta^n - \theta^{n - 1} $
in \eqref{Theta1_4} which yields
\[
\begin{split}
\frac{1}{k_n} \norm{\theta^n - \theta^{n - 1}}^2 =&
- \frac{i}{2} \left( \norm{\nabla \theta^n}^2 - \norm{\nabla \theta^{n- 1}}^2 \right)
+ 2 i \Im (\nabla \theta^n, \nabla \theta^{n - 1} ) \\
&- \left( \frac{\rho^n-\rho^{n-1}}{k_n }, \theta^n - \theta^{n - 1} \right) 
+ i ( \omega^n, \theta^n - \theta^{n - 1}) \\
&- ( r^n, \theta^n - \theta^{n - 1}).
\end{split}
\]
Taking the imaginary part of the above relation leads to
\[
\begin{split}
\frac{1}{2} \left( \norm{\nabla \theta^n}^2 - \norm{\nabla \theta^{n- 1}}^2 \right)=&
 - \Im \left( \frac{\rho^n-\rho^{n-1}}{k_n }, \theta^n - \theta^{n - 1} \right) 
+ \Re ( \omega^n, \theta^n - \theta^{n - 1}) \\
&- \Im ( r^n, \theta^n - \theta^{n - 1}).
\end{split}
\]
From the Cauchy-Schwarz inequality and the triangle inequality it follows that
\[
\norm{\nabla \theta^n}^2 - \norm{\nabla \theta^{n- 1}}^2 \leq
2 \left( \norm{ \theta^n} + \norm{ \theta^{n- 1}} \right)
\left(\fitnorm{\frac{\rho^n - \rho^{n - 1}}{k_n}} + \norm{\omega^n} + \norm{r^n} \right).
\]
Notice that we already have estimated the terms in the second factor
on the right hand side. For the $ \theta $-terms we use Poincare's inequality.
We therefore have the inequality
\[
\norm{\nabla \theta^n}^2 - \norm{\nabla \theta^{n- 1}}^2 \leq
C ( k^2 + h^r) \left( \norm{\nabla \theta^n} + \norm{\nabla \theta^{n- 1}} \right).
\]
Canceling common factors yields
\[
\norm{\nabla \theta^n} \leq \norm{\nabla \theta^{n - 1}}
+ C ( k^2 + h^r).
\]
Applying the above inequality iteratively leads to
\begin{equation} \label{gradtheta1_4}
\norm{\nabla \theta^n} \leq C(k + h^{r - 1})
\end{equation}
under the assumption that $ k $ is propotional to $ h $.
The desired estimate for $ \norm{\nabla e^n} $ now
follows from \eqref{gradtheta1_4} and \eqref{Ellip1_4}.

\end{proof}

\begin{remark}

Regarding the estimate of $ \norm{\nabla e^n} $,
one would expect an order of $ \mathcal{O} (k^2 + h^{r - 1}) $
since we only have spatial derivatives (cf. \cite{Zouraris:2001_4}).
However, in \cite{Yves:91_4} (as in here) an \textit{optimal} error estimate for the
$ H^1 $-norm of order $ \mathcal{O} (k + h^{r - 1}) $ is derived.
On the other hand compared to the gradient estimate by
Zouraris \cite{Zouraris:2001_4}, our $ H^1 $-norm estimate is suboptimal.

\end{remark}


\section{Concluding remarks}
In this note we considered discretizing a nonlinear Schr\"odinger equation 
based on a two-level time stepping scheme with an underlying finite element 
spatial discretization. The nonlinearity is of cubic type with crucial 
applications in, e.g., plasma physics, nonlinear optics and oceanography. 
In the spatial discretization we follow a strategy by Akrivis and 
co-workers \cite{Akrivis_etal:91_4} 
which rely on classical estimates due to 
Schatz-Wahlbin \cite{Al_Lars:82_4} and Thom\'{e}e \cite{Vidar_4}. 
We also, briefly, outline a more abstract form of a time-space scheme 
by \cite{Yves:91_4} and its convergence properties derived by Lopez-Marcos and 
Sanz-Serna \cite{Lopez_Sanz-Serna:91_4}. 
The crucial steps in the spatial 
discretization include the split of the error 
via $L_2$, $H^1$ and elliptic projections and then proceed with the argument 
of dominating the error between approximation and projection with that of the 
projection error (error between exact solution and the 
corresponding projection). Here both solution and gradient estimates 
are derived. 

Then these results are further used in half-steps in time 
following the results by Zouraris \cite{Zouraris:2001_4} and the references 
therein. In this part we prove 
the consistency of the numerical scheme, derive stability estimates and 
establish the convergence analysis.  In the  
temporal discretization, we have considered
the $L_\infty(L_2)$ approximations. 
The $L_2$ results are optimal of accuracy ${\mathcal O}(k^2+h^r)$, 
 due to the maximal available regularity of the exact solution, 
where $h$ and $k$ are spatial and temporal mesh parameters, respectively. 
As for the gradient estimates we prove convergence of order ${\mathcal O}(k+h^{r-1})$.
A more elaborate and different approach is the subject of a 
forthcoming paper \cite{Asadzadeh_Standar_Zouraris_4}.

\end{document}